%% file: IntegralMT.tex
\documentclass[a4paper,11pt,leqno]{article}

\usepackage[english]{babel}
\usepackage{amsmath,amsthm,amssymb,thmtools,enumitem,mathrsfs,mathtools,hyperref,url,doi,titlesec}
\usepackage[left=2cm,right=2cm,top=2.5cm,bottom=2.5cm,bindingoffset=0cm]{geometry}

\setlist[enumerate]{topsep=0pt,label=\textup{(\roman*)},leftmargin=\parindent,labelsep=.5em}
\setlist{noitemsep}

\usepackage{quoting}
\quotingsetup{vskip=0pt}

\usepackage[T1]{fontenc}

\usepackage{tikz}

\declaretheoremstyle[
  spaceabove=\topsep, spacebelow=6pt,
  headfont=\normalfont\bfseries,
  notefont=\mdseries, notebraces={(}{)},
  bodyfont=\normalfont\itshape,
  postheadspace=1em,
  qed=\qedsymbol
]{mystyle}

\declaretheoremstyle[
  spaceabove=\topsep, spacebelow=6pt,
  headfont=\normalfont\bfseries,
  notefont=\mdseries, notebraces={(}{)},
bodyfont=\normalfont,
  postheadspace=1em,
  qed=\qedsymbol
]{mydefstyle}

\swapnumbers
\theoremstyle{mystyle}
\declaretheorem[numberlike=subsection]{proposition}
\declaretheorem[numberlike=subsection]{theorem}
\declaretheorem[numberlike=subsection]{corollary}
\declaretheorem[numberlike=subsection]{lemma}

\declaretheorem[numberlike=subsection,name=Mumford-Tate Conjecture]{mtc}
\declaretheorem[numberlike=subsection,name=Integral Mumford-Tate Conjecture]{imtc}
\declaretheorem[numberlike=subsection,name=Adelic Mumford-Tate Conjecture]{amtc}
\declaretheorem[numbered=no,name=Theorem A]{MainA}
\declaretheorem[numbered=no,name=Theorem B]{MainB}

\theoremstyle{mydefstyle}
\declaretheorem[numberlike=subsection]{definition}
\declaretheorem[numberlike=subsection]{example}
\declaretheorem[numberlike=subsection]{remark}
\declaretheorem[numberlike=subsection]{remarks}

\numberwithin{equation}{subsection}

\titleformat{\section}[block]
  {\filcenter\normalfont\large\bfseries}{\thesection.}{1em}{}
\titleformat{\subsection}[runin]
  {\normalfont\bfseries}{\thesubsection}{.5em}{}

\titlespacing*{\section} {0pt}{6ex plus 1ex minus .2ex}{3 ex plus .2ex}
\titlespacing*{\subsection} {0pt}{\topsep}{1em}

\input{BMmacros}

\begin{document}

\begin{center}
\textbf{\Large Integral and adelic aspects of the Mumford-Tate conjecture}
\bigskip

\textit{by}
\bigskip

{\Large Anna Cadoret and Ben Moonen}
\end{center}
\vspace{1cm}

{\small %% for abstract and MSC

\noindent
\begin{quoting}
\textbf{Abstract.} Let $Y$ be an abelian variety over a subfield $k \subset \bbC$ that is of finite type over~$\bbQ$. We prove that if the Mumford-Tate conjecture for~$Y$ is true, then also some refined integral and adelic conjectures due to Serre are true for~$Y$. In particular, if a certain Hodge-maximality condition is satisfied, we obtain an adelic open image theorem for the Galois representation on the (full) Tate module of~$Y$. Our second main result is an (unconditional) adelic open image theorem for K3 surfaces. The proofs of these results rely on the study of a natural representation of the fundamental group of a Shimura variety.   
\medskip

\end{quoting}

} %% end of small
\vspace{1cm}

\input{Sect0-Intro}
\input{Sect1-IMTC}
\input{Sect2-HmaxAMTC}
\input{Sect3-AdelRepShim}

\input{Sect4-GalGen}
\input{Sect5-AV}
\input{Sect6-K3}

%%
%% make bibliography small:
%%
{\small

\bigskip

} %% end of small for the blibliography

\noindent
\texttt{anna.cadoret@math.polytechnique.fr}

\noindent
Centre de Math\'ematiques Laurent Schwartz -- \'Ecole Polytechnique, 91128 Palaiseau, France
\medskip

\noindent
\texttt{b.moonen@science.ru.nl}

\noindent
Radboud University Nijmegen, IMAPP, Nijmegen, The Netherlands

\end{document}

%% file: BMmacros
\newcommand{\bbA}{\mathbb{A}}

\newcommand{\bbC}{\mathbb{C}}

\newcommand{\bbF}{\mathbb{F}}
\newcommand{\bbG}{\mathbb{G}}
\newcommand{\bbH}{\mathbb{H}}

\newcommand{\bbQ}{\mathbb{Q}}
\newcommand{\bbR}{\mathbb{R}}
\newcommand{\bbS}{\mathbb{S}}

\newcommand{\bbZ}{\mathbb{Z}}

\newcommand{\cA}{\mathscr{A}}

\newcommand{\cC}{\mathscr{C}}

\newcommand{\cF}{\mathscr{F}}
\newcommand{\cG}{\mathscr{G}}

\newcommand{\cS}{\mathscr{S}}

\newcommand{\Betti}{\mathrm{B}}
\newcommand{\Coker}{\mathrm{Coker}}
\newcommand{\CSp}{\mathrm{GSp}}
\newcommand{\CSpin}{\mathrm{CSpin}}
\newcommand{\End}{\mathrm{End}}
\newcommand{\Gal}{\mathrm{Gal}}
\newcommand{\GL}{\mathrm{GL}}
\newcommand{\GO}{\mathrm{GO}}
\newcommand{\Image}{\mathrm{Im}}
\newcommand{\Inn}{\mathrm{Inn}}
\newcommand{\Ker}{\mathrm{Ker}}
\newcommand{\Norm}{\mathrm{Norm}}
\newcommand{\Res}{\mathrm{Res}}
\newcommand{\Sh}{\mathrm{Sh}}
\newcommand{\SO}{\mathrm{SO}}
\newcommand{\Spec}{\mathrm{Spec}}
\newcommand{\UU}{\mathrm{U}}

\newcommand{\ab}{\mathrm{ab}}
\newcommand{\ad}{\mathrm{ad}}
\newcommand{\conn}{\mathrm{conn}}
\newcommand{\der}{\mathrm{der}}
\newcommand{\diag}{\mathrm{diag}}
\newcommand{\hgen}{\mathrm{Hgen}}
\newcommand{\id}{\mathrm{id}}

\newcommand{\mult}{\mathrm{m}}
\newcommand{\pr}{\mathrm{pr}}
\newcommand{\prim}{\mathrm{prim}}
\newcommand{\rec}{\mathrm{rec}}
\newcommand{\trace}{\mathrm{trace}}

\newcommand{\Af}{\bbA_\mathit{f}}
\newcommand{\kbar}{{\bar{k}}}
\newcommand{\Qbar}{{\overline{\bbQ}}}
\newcommand{\Zhat}{\hat{\bbZ}}

\newcommand{\QHS}{{\bbQ\mathsf{HS}}}

\newcommand{\isomarrow}{\xrightarrow{\sim}}

\let\epsilon=\varepsilon

%% file: Sect0-Intro
\section*{Introduction}

This paper has two main results, one about abelian varieties, the other about K3 surfaces. The underlying mechanism is the same in both cases.

To explain the results, let $Y$ be either an abelian variety or a K3 surface over a subfield $k \subset \bbC$ that is finitely generated over~$\bbQ$. Write $H = H_1\bigl(Y(\bbC),\bbZ\bigr)$ in the first case and $H = H^2\bigl(Y(\bbC),\bbZ\bigr)(1)$ if $Y$ is a K3 surface. Let $G_\Betti \subset \GL(H)$ be the Mumford-Tate group. We may identify $H \otimes \Zhat$ with the \'etale cohomology of~$Y$ with $\Zhat$-coefficients ($H_1$ or $H^2(1)$); this gives us a Galois representation
\[
\rho_Y \colon \Gal(\kbar/k) \to \GL(H)\bigl(\Zhat\bigr)\, .
\]
It is known that, possibly after replacing~$k$ with a finite extension, the image of~$\rho_Y$ is contained in~$G_\Betti(\Zhat)$. {}From now on, we assume this is the case. If $\ell$ is a prime number, let $\rho_{Y,\ell} \colon \Gal(\kbar/k) \to \GL(H)\bigl(\bbZ_\ell\bigr)$ be the $\ell$-primary component of~$\rho_Y$.

In the case of an abelian variety, our main result says that the usual Mumford-Tate conjecture implies an integral refinement of it. If moreover a certain maximality condition is satisfied (see Section~\ref{sec:HmaxAMTC}), we obtain an adelic open image result for~$\rho_Y$. These refined statements were conjectured by Serre. The precise result is as follows.

\begin{MainA}
Let $Y$ be an abelian variety over~$k$ for which the Mumford-Tate conjecture is true.

\textup{(\romannumeral1)} The index $\bigl[G_\Betti(\bbZ_\ell):\Image(\rho_{Y,\ell})\bigr]$ is bounded when $\ell$ varies. Moreover, for almost all~$\ell$ the image of~$\rho_{Y,\ell}$ contains the commutator subgroup of~$G_\Betti(\bbZ_\ell)$, as well as the integral homotheties $\bbZ_\ell^* \cdot \id$.

\textup{(\romannumeral2)} If the Hodge structure $H_1\bigl(Y(\bbC),\bbQ\bigr)$ is Hodge-maximal (see \textup{Def.~\ref{def:HodgeMax}}), the image of~$\rho_Y$ is an open subgroup of~$G_\Betti(\Af)$. 
\end{MainA}

By a result of Larsen and Pink (see \cite{LaPi1995}, Thm.~4.3), if the Mumford-Tate conjecture for an abelian variety is true for one prime number~$\ell$, it is true for all~$\ell$. The assumption that the Mumford-Tate conjecture for~$Y$ is true is therefore unambiguous. Let us also note that Hodge-maximality is a necessary condition for the image of~$\rho_Y$ to be open in~$G_\Betti(\Af)$; see Remark~\ref{rem:HMaxWintenb}.

In the case of a K3 surface, the Mumford-Tate conjecture is known (due to Tankeev~\cite{TankK3} and, independently, Andr\'e~\cite{AndreK3}), and we prove that the Hodge-maximality assumption is always satisfied. In this case we obtain the following adelic open image theorem.

\begin{MainB}
If $Y$ is a K3 surface, the image of~$\rho_Y$ is an open subgroup of~$G_\Betti(\Af)$.
\end{MainB}

The proofs of these theorems rely on the fact that the moduli spaces of abelian varieties and K3 surfaces are (essentially) Shimura varieties. 

In Section~\ref{sec:AdelicRep} we recall how to associate to a Shimura datum $(G,X)$ and a neat compact open subgroup $K_0 \subset G(\Af)$ a representation $\phi \colon \pi_1(S_0) \to K_0 \subset G(\Af)$, where $S_0 \subset \Sh_{K_0}(G,X)$ (over some number field~$F$) is a geometrically irreducible component. This reflects the fact that we have a tower of finite \'etale covers $\Sh_K(G,X) \to \Sh_{K_0}(G,X)$ indexed by the open subgroups $K \subset K_0$. For $\ell$ a prime number, let $\phi_\ell \colon \pi_1(S_0) \to K_{0,\ell}$ be the $\ell$-primary component of~$\phi$.

The main technical result we prove, as Corollary~\ref{cor:Im(phi)}, is that, under a mild assumption on the datum $(G,X)$, the image of the adelic representation~$\phi$ is ``big'' in precisely the sense as in the conclusions of the main theorems. The proof of this result only involves abstract theory of Shimura varieties; it relies on Deligne's group-theoretic description of the reciprocity law that gives the Galois action on the geometric connected components of $\Sh(G,X)$.

To understand how this result leads to our main theorems, we have to link the representation~$\phi$ to the Galois representation~$\rho_Y$ on the cohomology of~$Y$. Taking $G$ to be the Mumford-Tate group of~$Y$, we obtain a Shimura variety that has an interpretation as a moduli space of abelian varieties or K3 surfaces with additional structures. If $y \in S_0(k)$ is the point corresponding to~$Y$, we obtain a homomorphism $\sigma_y\colon \pi_1(y) \to \pi_1(S_0)$ (with $\pi_1(y) \cong \Gal(\kbar/k)$), and the composition $\phi\circ \sigma_y$ is isomorphic to the Galois representation~$\rho_Y$. Using a classical result of Bogomolov, the assumption that the Mumford-Tate conjecture for~$Y$ is true, for some prime number~$\ell$, implies that the image of $\phi_\ell\circ \sigma_y$ is open in the image of~$\phi_\ell$. Points $y \in S_0(k)$ for which this holds are said to be $\ell$-Galois-generic with respect to~$\phi$. For points on a Shimura variety of abelian type, a result of the first author and Kret says that being $\ell$-Galois-generic for some~$\ell$ implies something that is a priori much stronger, namely that the image of the adelic representation $\phi\circ \sigma_y$ is open in the image of~$\phi$. (See \cite{CadKret}, Thm.~1.1. This result is a consequence of the open adelic image theorem for abelian schemes proven by the first author in~\cite{Cadoret}.) Our main results are obtained by combining this with our result on the image of~$\phi$.

The first two sections of the paper are of a preliminary nature. We recall some conjectures due to Serre that refine the Mumford-Tate conjecture. Also we discuss the notion of (Hodge-)maximality, which is a necessary condition for an adelic open image theorem to hold. Section~\ref{sec:AdelicRep} forms the core of the paper. We define the representation~$\phi$ associated with a Shimura variety; further we state and prove the main result, Theorem~\ref{thm:Coker(rec)} and its Corollary~\ref{cor:Im(phi)}, about the image of~$\phi$. In Section~\ref{sec:GalGen} we briefly recall various notions of Galois-genericity, and we state the result of Cadoret-Kret that we need. In Section~\ref{sec:AV}, which is devoted to abelian varieties, we prove Theorem~A. Also we give examples of abelian varieties for which the~$H_1$ is not Hodge-maximal. These examples suggest that Hodge-maximality depends in a rather subtle way on the structure of the Mumford-Tate group. Finally, in Section~\ref{sec:K3} we discuss K3 surfaces. We prove that the $H^2(1)$ of a K3 surface is always Hodge-maximal, and we deduce Theorem~B.

\subsection*{Acknowledgement.} We thank Akio Tamagawa for his interest and for helpful discussions.

%% file: Sect1-IMTC
\section{An integral variant of the Mumford-Tate conjecture}

\subsection{}
\label{ssec:H(Y)}
Let $Y$ be a smooth proper scheme of finite type over a subfield~$k$ of~$\bbC$ that is finitely generated over~$\bbQ$. Fix integers $i$ and~$n$.

Let $H = H^i\bigl(Y(\bbC),\bbZ\bigr)(n)/(\text{torsion})$, which is a polarizable Hodge structure of weight $i-2n$. We denote by $G_\Betti \subset \GL(H)$ the Mumford-Tate group. By this we mean that the generic fibre $G_{\Betti,\bbQ} \subset \GL(H_\bbQ)$ is the Mumford-Tate group of~$H_\bbQ$ in the usual sense, and that $G_\Betti$ is the Zariski closure of~$G_{\Betti,\bbQ}$ inside~$\GL(H)$. If the context requires it, we include $Y$ in the notation, writing $G_{\Betti,Y}$, etc.

For a prime number~$\ell$, let $H_\ell = H^i\bigl(Y_\kbar,\bbZ_\ell\bigr)(n)/(\text{torsion})$, which is a free  $\bbZ_\ell$-module of finite rank on which we have a continuous Galois representation
\[
\rho_{\ell} \colon \Gal(\kbar/k) \to \GL(H_\ell)\, .
\]
We denote by $G_\ell \subset \GL(H_\ell)$ the Zariski closure of the image of~$\rho_\ell$. The generic fibre $G_{\ell,\bbQ_\ell} \subset \GL(H_{\ell,\bbQ_\ell})$ is the Zariski closure of the image of the Galois representation on~$H_{\ell,\bbQ_\ell}$. We define $G_\ell^0 \subset G_\ell$ to be the Zariski closure of the identity component $(G_{\ell,\bbQ_\ell})^0$.

If we replace~$k$ by a finitely generated extension, $G_\ell$ may become smaller, but its identity component~$G_\ell^0$ does not change. By a result of Serre (see \cite{SerreRibet} or \cite{LaPi1992}, Prop.~6.14), there exists a finite field extension $k \subset k^\conn$ in~$\bbC$ (depending on~$Y$, $i$ and~$n$) such that for every field $K$ that contains~$k^\conn$ and every prime number~$\ell$, the generic fibre of~$G_{\ell,Y_K}$ is connected.

Via the comparison isomorphism $H \otimes \bbZ_\ell \isomarrow H_\ell$, we may view $G_\Betti \otimes \bbZ_\ell$ as a subgroup scheme of $\GL(H_\ell)$.

\begin{mtc}
\label{conj:MTC}
With notation as above, $G_\Betti \otimes \bbZ_\ell = G_\ell^0$ as subgroup schemes of $\GL(H_\ell)$.
\end{mtc}

Note that, though we have stated the conjecture using group schemes over~$\bbZ_\ell$, the Mumford-Tate conjecture in this form is equivalent to the conjecture that $G_\Betti \otimes \bbQ_\ell$ equals $G^0_{\ell,\bbQ_\ell}$ as algebraic subgroups of $\GL(H_{\ell,\bbQ_\ell})$, which is the Mumford-Tate conjecture as it is usually stated. As $H_\ell$ is Hodge-Tate, it follows from a result of Bogomolov~\cite{Bogomolov} (with some extensions due to Serre; see also~\cite{SerreRibet}) that the image of~$\rho_\ell$ is open in~$G_\ell(\bbQ_\ell)$. Hence the Mumford-Tate conjecture is equivalent to the assertion that $\Image(\rho_\ell)$ is an open subgroup of~$G_\Betti(\bbZ_\ell)$ (assuming $k=k^\conn$). Further note that the Mumford-Tate conjecture depends, a priori, on~$\ell$ and also on the chosen complex embedding of~$k$.

The following strengthening of the Mumford-Tate conjecture was proposed by Serre; see Conjecture~C.3.7 in~\cite{Serre112} and cf.~\cite{SerreMotives}.

\begin{imtc}[Serre]
\label{conj:IMTC}
Retain the notation of~\textup{\ref{ssec:H(Y)}}, and assume $k=k^\conn$. Then for all~$\ell$ the image $\Image(\rho_\ell)$ is contained in~$G_\Betti(\bbZ_\ell)$ as an open subgroup, and the index $\bigl[G_\Betti(\bbZ_\ell) : \Image(\rho_\ell)\bigr]$ is bounded when $\ell$ varies. Further, for almost all~$\ell$ the image of~$\rho_\ell$ contains the commutator subgroup of~$G_\Betti(\bbZ_\ell)$ and all homotheties of the form $c^{i-2n}\cdot \id$, for $c \in \bbZ_\ell^*$.
\end{imtc}

Compared with the usual Mumford-Tate conjecture, the main point in the above conjecture is that it should be possible to bound the index of $\Image(\rho_\ell)$ in~$G_\Betti(\bbZ_\ell)$ by a constant independent of~$\ell$.  

%% file: Sect2-HmaxAMTC
\section{Maximality, and an adelic form of the Mumford-Tate conjecture}
\label{sec:HmaxAMTC}

\begin{definition}
\label{def:Maximal}
Let $M$ be a connected algebraic group over a field~$k$ of characteristic~$0$. Let $k \subset F$ be a field extension, $S$ an algebraic group over~$F$, and $h \colon S \to M_F$ a homomorphism. Then we say that $h$ is \emph{maximal} if there is no non-trivial isogeny of connected $k$-groups $M^\prime \to M$ such that $h$ lifts to a homomorphism $S \to M^\prime_F$.
\end{definition}

Note that if $F$ is algebraically closed, maximality of~$h$ only depends on its $M(F)$-conjugacy class.

\subsection{}
\label{ssec:pi1Borov}
The following remarks closely follow \cite{WintenbCM},~0.2. Let $M$ be a connected reductive group over a subfield $k \subset \bbC$. Let $\cC$ be a conjugacy class of complex cocharacters $\mu \colon \bbG_{\mult,\bbC} \to M_\bbC$. Let $\pi_1(M)$ denote the fundamental group of~$M$ as defined by Borovoi in~\cite{Borovoi}. This is a finitely generated $\bbZ$-module with a continuous action of $\Gamma = \Gal(\kbar/k)$. If $(X^*,R,X_*,R^\vee)$ is the root datum of~$M_\kbar$ and $Q(R^\vee) = \langle R^\vee\rangle \subset X_*$ is the coroot lattice, $\pi_1(M) \cong X_*/Q(R^\vee)$. 

The conjugacy class~$\cC$ of complex cocharacters corresponds to an orbit $\cC \subset X_*$ under the Weyl group~$W$. As the induced $W$-action on~$\pi_1(M)$ is trivial, any two elements in~$\cC$ have the same image in~$\pi_1(M)$; call it $[\cC] \in \pi_1(M)$.

If $M^\prime$ is a connected reductive $k$-group and $f\colon M^\prime \to M$ is an isogeny, the map induced by~$f$ identifies $\pi_1(M^\prime)$ with a $\bbZ[\Gamma]$-submodule of finite index in~$\pi_1(M)$. Conversely, every such submodule comes from an isogeny of connected $k$-groups, which is unique up to isomorphism over~$M$. A conjugacy class~$\cC$ as above lifts to~$M^\prime$ if and only if $[\cC] \in \pi_1(M^\prime)$. 

We shall usually be in a situation where the $\bbZ[\Gamma]$-submodule spanned by~$[\cC]$ has finite index in~$\pi_1(M)$. (See below.) In this case, there is a uniquely determined maximal isogeny $M^\prime \to M$ of connected $k$-groups such that the $\mu \in \cC$ lift to complex cocharacters of~$M^\prime$. Further, the cocharacters $\mu \in \cC$ are maximal in the sense of Definition~\ref{def:Maximal} if and only if $[\cC]$ generates~$\pi_1(M)$ as a $\bbZ[\Gamma]$-module.

\begin{definition}
\label{def:HodgeMax}
Let $V$ be a $\bbQ$-Hodge structure, given by the homomorphism $h \colon \bbS \to \GL(V)_\bbR$. Let $M \subset \GL(V)$ be the Mumford-Tate group. Then $V$ is said to be \emph{Hodge-maximal} if $h \colon \bbS \to M_\bbR$ is maximal in the sense of Definition~\ref{def:Maximal}.
\end{definition}

Hodge-maximality of~$V$ is equivalent to the condition that the associated cocharacter $\mu \colon \bbG_{\mult,\bbC} \to M_\bbC$ is maximal. This allows us to apply~\ref{ssec:pi1Borov}, taking $k=\bbQ$ and $\Gamma = \Gal(\Qbar/\bbQ)$. If $\cC$ is the $M(\bbC)$-conjugacy class of~$\mu$, the assumption that $M$ is the Mumford-Tate group of~$V$ implies that the $\bbZ[\Gamma]$-submodule of~$X_*(M)$ generated by~$\cC$ has finite index. Hence also the $\bbZ[\Gamma]$-submodule of~$\pi_1(M)$ generated by~$[\cC]$ has finite index. By what was explained in~\ref{ssec:pi1Borov}, $V$ is Hodge-maximal if and only if $\bbZ[\Gamma] \cdot [\cC] = \pi_1(M)$.

\subsection{}
\label{ssec:AdelicRho}
Retaining the notation and assumptions of~\ref{ssec:H(Y)}, let $\hat{H} = \prod_{\ell}\; H_\ell$, where the product is taken over all prime numbers~$\ell$. We then have a continuous Galois representation
\[
\rho \colon \Gal(\kbar/k) \to \GL(\hat{H})
\]
whose $\ell$-primary component is the representation~$\rho_\ell$ defined in~\ref{ssec:H(Y)}. The comparison isomorphism between singular and \'etale cohomology gives an isomorphism $\hat{H} \cong H \otimes \Zhat$; via this we may view~$\rho$ as a representation taking values in $\GL(H)\bigl(\Zhat\bigr)$.

The following adelic version of the Mumford-Tate conjecture was proposed by Serre; see Conjecture~C.3.8 in~\cite{Serre112}.

\begin{amtc}[Serre]
\label{conj:AMTC}
With notation as in~\textup{\ref{ssec:H(Y)}}, suppose that $k=k^\conn$ and that the Hodge structure~$H$ is Hodge-maximal. Then $\Image(\rho)$ is an open subgroup of~$G_\Betti(\Af)$.
\end{amtc}

\begin{remark}
\label{rem:HMaxWintenb}
It follows from a result of Wintenberger that the Hodge-maximality of~$H$ is essential. (For simplicity we shall assume here that the ground field~$k$ is a number field.) Indeed, suppose there exists a non-trivial isogeny of $\bbQ$-groups $M^\prime \to G_\Betti$ with $M^\prime$ connected, such that $h \colon \bbS \to G_{\Betti,\bbR}$ lifts to a homomorphism $\bbS \to M^\prime_\bbR$. By \cite{WintenbRelev}, Th\'eor\`eme~2.1.7, and possibly after replacing the ground field~$k$ with a finite extension, the $\ell$-adic representations~$\rho_\ell$ lift to Galois representations with values in~$M^\prime$. On the other hand, it follows from \cite{PlatonRapin}, Proposition~6.4, that the image of $M^\prime(\Af) \to G_\Betti(\Af)$ is not open in~$G_\Betti(\Af)$; hence $\Image(\rho)$ cannot be open in~$G_\Betti(\Af)$.  
\end{remark}

\subsection{}
\label{ssec:HmaxGX}
Let $(G,X)$ be a Shimura datum such that $G$ is the generic Mumford-Tate group on~$X$. By definition, this means that there exist points $h \in X$ for which there is no proper subgroup $G^\prime \subset G$ such that $h\colon \bbS \to G_\bbR$ factors through~$G^\prime_\bbR$. The locus of points~$h$ for which this holds forms a subset $X^\hgen \subset X$ called the Hodge-generic locus.

Similar to the definition in~\ref{def:HodgeMax}, we say that $(G,X)$ is maximal if there is no non-trivial isogeny of Shimura data $f\colon (G^\prime,X^\prime) \to (G,X)$. (Note that $G^\prime$ is necessarily connected, as it is part of a Shimura datum.) Clearly, if $(G,X)$ is maximal then all $h \colon \bbS \to G_\bbR$ in~$X^\hgen$ are maximal in the sense of Definition~\ref{def:Maximal}. Conversely, if some $h\in X^\hgen$ is maximal then $(G,X)$ is maximal. (If we have $f$ as above, $f(X^\prime) \subset X$ is a union of connected components but need not be the whole~$X$; however, changing~$f$ by an inner automorphism of~$G$ we can always ensure that some given $h \in X$ lies in~$f(X^\prime)$.)

To each $h \in X$ corresponds a complex cocharacter~$\mu_h$ of~$G$, and the $\mu_h$ thus obtained all lie in a single $G(\bbC)$-conjugacy class~$\cC(G,X)$. It follows from the previous remarks that $(G,X)$ is maximal if and only if the associated class $\bigl[\cC(G,X)\bigr]$ generates~$\pi_1(G)$ as a $\bbZ[\Gamma]$-module.

\begin{remarks}
\label{rem:MorphShim}
(\romannumeral1) Let $f\colon (G_1,X_1) \to (G_2,X_2)$ be a morphism of Shimura data with $f\colon G_1 \to G_2$ surjective. If $G_1$ is the generic Mumford-Tate group on~$X_1$ then $G_2$ is the generic Mumford-Tate group on~$X_2$. If $f\colon G_1 \to G_2$ is an isogeny then also the converse is true.

(\romannumeral2) If in~(\romannumeral1) $\Ker(f)$ is semisimple then also maximality is preserved: if $(G_1,X_1)$ is maximal, so is $(G_2,X_2)$. Indeed, in this case $\pi_1(G_2)$ is a quotient of~$\pi_1(G_1)$ in such a way that $\bigl[\cC(G_2,X_2)\bigr]$ is the image of $\bigl[\cC(G_1,X_1)\bigr]$.

(\romannumeral3) Given a Shimura datum $(G,X)$, it follows from the remarks in~\ref{ssec:pi1Borov} that, up to isomorphism, there exists a unique isogeny of Shimura data $f \colon (\tilde{G},\tilde{X}) \to (G,X)$ such that $(\tilde{G},\tilde{X})$ is maximal. By~(\romannumeral1), $G$ is the generic Mumford-Tate group on~$X$ if and only if $\tilde{G}$ is the generic Mumford-Tate group on~$\tilde{X}$.
\end{remarks}

%% file: Sect3-AdelRepShim
\section{Adelic representations associated with Shimura varieties}
\label{sec:AdelicRep}

\subsection{}
\label{ssec:ShimRepDef}
Let $(G,X)$ be a Shimura datum. Throughout we assume that $G$ is the generic Mumford-Tate group on~$X$. (See~\ref{ssec:HmaxGX}.) In this case, conditions (2.1.1.1--5) of \cite{DelShimura} are satisfied and $Z(\bbQ)$ is discrete in~$Z(\Af)$. (Cf.\ \cite{DelShimura}, 2.1.11; for details see also \cite{UllYafQuebec}, Lemma~5.13.) 

If $K \subset G(\Af)$ is a compact open subgroup, we have $\Sh_K(G,X)\bigl(\bbC\bigr) = G(\bbQ)\backslash X \times G(\Af)/K$. For $h \in X$ and $\gamma K \in G(\Af)/K$, let $[h,\gamma K]$ denote the corresponding $\bbC$-valued point of $\Sh_K(G,X)$.

Let $K_0 \subset G(\Af)$ be a neat compact open subgroup. If $K \subset K_0$ is an open subgroup, the induced morphism on Shimura varieties $\Sh_{K,K_0} \colon \Sh_K(G,X) \to \Sh_{K_0}(G,X)$ is finite \'etale. If, moreover, $K$ is normal in~$K_0$, this morphism is Galois with group $K_0/K$.

Choose a point $h_0 \in X$. Let $S_{0,\bbC}$ be the irreducible component of $\Sh_{K_0}(G,X)_\bbC$ that contains the point $[h_0,eK_0]$. Let $F$ be the field of definition of this component, which is a finite extension of the reflex field~$E(G,X)$. To simplify notation, we write $\Sh_K$ for $\Sh_K(G,X)_F$. By construction, we have a geometrically irreducible component $S_0 \subset \Sh_{K_0}$.

For $K$ an open normal subgroup of~$K_0$, let $S_K \subset \Sh_K$ be the inverse image of $S_0 \subset \Sh_{K_0}$ under the transition morphism $\Sh_{K,K_0}$. Then $S_K \to S_0$ is \'etale Galois with group $K_0/K$. 

Let $\bar{s}_K = [h_0,eK] \in S_K(\bbC)$. The system of points $\bar{s} = (\bar{s}_K)$ thus obtained is compatible in the sense that $\Sh_{K_2,K_1}(\bar{s}_{K_2}) = \bar{s}_{K_1}$ for $K_2 \subset K_1 \subset K_0$. We abbreviate $\bar{s}_{K_0}$ to~$\bar{s}_0$. With this choice of base points, $S_K \to S_0$ corresponds to a homomorphism $\phi_K \colon \pi_1(S_0,\bar{s}_0) \to K_0/K$. If $K_2 \subset K_1$ are open normal subgroups of~$K_0$, the homomorphism $\phi_{K_1}$ equals the composition of $\phi_{K_2} \colon \pi_1(S_0,\bar{s}_0) \to K_0/K_2$ and the canonical map $K_0/K_2 \to K_0/K_1$. We may therefore pass to the limit; as the intersection of all open normal subgroups $K \subset K_0$ is trivial, this gives a continuous homomorphism
\begin{equation}
\label{eq:phi}
\phi_{\bar{s}} \colon \pi_1(S_0,\bar{s}_0) \to K_0\, .
\end{equation}

\begin{remarks}
\label{rem:phi}
(\romannumeral1) The homomorphism~$\phi$ is functorial in the following sense. Let $f \colon (G,X) \to (G^\prime,X^\prime)$ be a morphism of Shimura data. On reflex fields we have $E(G^\prime,X^\prime) \subset E = E(G,X)$. Let $K_0 \subset G(\Af)$ and $K_0^\prime \subset G^\prime(\Af)$ be neat compact open subgroups with $f(K_0) \subset K_0^\prime$. Choose $h_0 \in X$ and let $h_0^\prime = f(h_0) \in X^\prime$. As in~\ref{ssec:ShimRepDef}, this gives rise to geometrically irreducible components $S_0 \subset \Sh_{K_0}(G,X)_F$ and $S_0^\prime \subset \Sh_{K_0^\prime}(G^\prime,X^\prime)_{F^\prime}$, and it is easy to see that $EF^\prime \subset F$. Further, $h_0$ and~$h_0^\prime$ give rise to compatible systems of base points $\bar{s} = (\bar{s}_K)$ and $\bar{s}^\prime = (\bar{s}^\prime_{K^\prime})$. The morphism $\Sh(f) \colon \Sh_{K_0}(G,X) \to \Sh_{K_0^\prime}(G^\prime,X^\prime)_E$ restricts to a morphism $S_0 \to S_{0,F}^\prime$ over~$F$ with $\bar{s}_0 \mapsto \bar{s}_0^\prime$. We then have a commutative diagram
\[
\begin{tikzpicture}
\node (A) at (0,2) {$\pi_1(S_0,\bar{s}_0)$};
\node (B) at (4.5,1.95) {$\pi_1(S_{0,F}^\prime,\bar{s}_0^\prime) \subset \pi_1(S_0^\prime,\bar{s}_0^\prime)$};
\node (C) at (0,0) {$K_0$};
\node (D) at (5.5,0) {$K_0^\prime$};
\node (E) at (5.5,1.8) {};
\draw[->] (A) -- (C) node[pos=.5,left] {$\scriptstyle \phi_{\bar{s}}$};
\draw[->] (A) -- (B) node[pos=.5,above] {$\scriptstyle \Sh(f)_*$};
\draw[->] (C) -- (D) node[pos=.5,above] {$\scriptstyle f$};
\draw[->] (E) -- (D) node[pos=.48,right] {$\scriptstyle \phi^\prime_{\bar{s}^\prime}$};
\end{tikzpicture}
\]

(\romannumeral2) The homomorphism~$\phi$ is essentially independent of the choice of $h_0 \in X$ and the resulting system of base points~$\bar{s}$. If we choose another point $h_0^\prime \in X$ that lies in the same connected component as~$h_0$, this gives rise to a different collection of base points $\bar{s}^\prime$. There is a canonically determined conjugacy class of isomorphisms $\alpha\colon \pi_1(S_0,\bar{s}_0) \isomarrow \pi_1(S_0,\bar{s}_0^\prime)$. For $\alpha$ in this class, the homomorphisms $\phi_{\bar{s}}$ and $\phi_{\bar{s}^\prime} \circ \alpha$ differ by an inner automorphism of~$K_0$.

If $h_0^\prime$ lies in a different connected component of~$X$, there exists an inner automorphism $\alpha = \Inn(g)$ of~$(G,X)$ such that $\alpha(h_0)$ and~$h_0^\prime$ lie in the same component of~$X$. By functoriality together with the previous case, it follows that the associated representations $\phi_{\bar{s}}$ and~$\phi_{\bar{s}^\prime}$ are conjugate when we restrict to suitable subgroups of finite index in the respective $\pi_1$'s.
\end{remarks}

In view of the above remarks, we shall from now on omit the base point~$\bar{s}_0$ from the notation, unless it plays a role in the discussion.

\subsection{}
Our main goal in this section is to describe the image of the representation $\phi = \phi_{\bar{s}}$ defined in~\ref{ssec:ShimRepDef}. In order to do this, we need to recall some definitions and results from the theory of Shimura varieties. For proofs of the stated results we refer to~\cite{DelShimura}, Section~2. Throughout, $(G,X)$ is a Shimura datum as in~\ref{ssec:ShimRepDef}. Let $G(\bbR)_+ \subset G(\bbR)$ be the subgroup of elements that are mapped into the identity component $G^\ad(\bbR)^+ \subset G^\ad(\bbR)$ (for the Euclidean topology) under the adjoint map. 

Let $\gamma \colon \tilde{G} \to G^\der$ denote the simply connected cover of the derived group of~$G$. Then $G(\bbQ)\gamma\tilde{G}(\bbA)$ is a closed normal subgroup of~$G(\bbA)$, and we define $\pi(G) = G(\bbA)/G(\bbQ)\gamma\tilde{G}(\bbA)$. Next define $\bar\pi_0\pi(G) = \pi_0\pi(G)/\pi_0G(\bbR)_+$, where $\pi_0$ means the group of connected components. This $\bar\pi_0\pi(G)$ is an abelian profinite group.

Define $G(\bbQ)_+ = G(\bbQ) \cap G(\bbR)_+$, and let $G(\bbQ)_+^-$ denote its closure inside~$G(\Af)$. The natural homomorphism $G(\Af) \to \bar\pi_0\pi(G)$ induces an isomorphism $G(\Af)/G(\bbQ)_+^- \isomarrow \bar\pi_0\pi(G)$. If there is no risk of confusion we identify the two groups. 

The group $G(\Af)$ acts on the Shimura variety $\Sh(G,X)$ from the right. This action makes the set $\pi_0\bigl(\Sh(G,X)_\bbC\bigr)$ a torsor under $\bar\pi_0\pi(G)$.

\subsection{}
\label{ssec:rechom}
Let $E = E(G,X)$ be the reflex field and $E^\ab$ its maximal abelian extension. As $\pi_0\bigl(\Sh(G,X)_\bbC\bigr)$ is a torsor under $\bar\pi_0\pi(G)$, which is abelian, the action of $\Gal(\overline{E}/E)$ on $\pi_0\bigl(\Sh(G,X)_\bbC\bigr)$ gives rise to a well-determined homomorphism
\begin{equation}
\label{eq:rec}
\rec\colon \Gal(E^\ab/E) \to \bar\pi_0\pi(G) \cong G(\Af)/G(\bbQ)_+^- \, ,
\end{equation}
called the reciprocity homomorphism.

Let $q\colon G(\Af) \to \bar\pi_0\pi(G)$ be the canonical map. For $K \subset G(\Af)$ a compact open subgroup, we have an induced action of $\bar\pi_0\pi(G)$ on the set of irreducible components of $\Sh_K(G,X)_\bbC$. All these components have the same stabilizer in $\bar\pi_0\pi(G)$, namely~$q(K)$. Let $\rec_K\colon \Gal(E^\ab/E) \to \bar\pi_0\pi(G)/q(K)$ denote the reciprocity map modulo~$q(K)$.

\begin{proposition}
\label{prop:Im(phi)Step1}
Retain the notation and assumptions of~\textup{\ref{ssec:ShimRepDef}}. Then the image of the homomorphism $\phi \colon \pi_1(S_0) \to K_0$ is the subgroup $q^{-1}\bigl(\Image(\rec)\bigr) \cap K_0$ of~$K_0$. 
\end{proposition}

\begin{proof}
Since $F$ is defined to be the field of definition of the irreducible component $S_{0,\bbC} \subset \Sh_{K_0}(G,X)_\bbC$, we have 
\[
\Gal(E^\ab/F) = \rec^{-1}\bigl(q(K_0)\bigr)
\]
as subgroups of~$\Gal(E^\ab/E)$. For $K$ a normal open subgroup of~$K_0$, the set of geometric irreducible components of~$S_K$ is a torsor under $q(K_0)/q(K) \subset \bar\pi_0\pi(G)/q(K)$, and the irreducible components of~$S_K$ (over~$F$) correspond to the orbits under the action of $\Gal(E^\ab/F)$ via~$\rec_K$. Hence the image of $\phi_K \colon \pi_1(S_0) \to K_0/K$ is the inverse image under $q \colon K_0/K \to q(K_0)/q(K)$ of $\rec_K\bigl(\Gal(E^\ab/F)\bigr)$. The latter group is the image of $\bigl(\Image(\rec) \cap q(K_0)\bigr)$ in $q(K_0)/q(K)$, and so we find that $\Image(\phi_K)$ is the image of $q^{-1}\bigl(\Image(\rec)\bigr) \cap K_0$ in $K_0/K$. The proposition follows by passing to the limit.
\end{proof}

\begin{theorem}
\label{thm:Coker(rec)}
With assumptions as in~\textup{\ref{ssec:ShimRepDef}}, the cokernel of the reciprocity map~\eqref{eq:rec} has finite exponent, and it is a finite discrete group (i.e., $\Image(\rec) \subset \bar\pi_0\pi(G)$ is an open subgroup) if $(G,X)$ is maximal. 
\end{theorem}

Before we start discussing the proof, let us give the main corollary of this result.

\begin{corollary}
\label{cor:Im(phi)}
With assumptions as in~\textup{\ref{ssec:ShimRepDef}}, consider the homomorphism $\phi \colon \pi_1(S_0) \to K_0$, and for a prime number~$\ell$, let $\phi_\ell \colon \pi_1(S_0) \to K_{0,\ell}$ be its $\ell$-primary component. Fix a free $\bbZ$-module~$H$ of finite rank and a closed embedding $i \colon G \hookrightarrow \GL(H\otimes \bbQ)$, and let $\cG \subset \GL(H)$ be the Zariski closure of~$G$ in $\GL(H)$.

\begin{enumerate}
\item There exists a positive integer~$N$, depending only on $\cG$ and~$X$, such that $\bigl[\cG(\bbZ_\ell):\Image(\phi_\ell)\bigr] \leq N$ for all~$\ell$.

\item For almost all~$\ell$ the image of~$\phi_\ell$ contains the commutator subgroup of~$\cG(\bbZ_\ell)$.

\item If $(G,X)$ is maximal in the sense defined in~\textup{\ref{ssec:HmaxGX}}, $\Image(\phi)$ is an open subgroup of~$G(\Af)$.
\end{enumerate}
\end{corollary}

\begin{proof}
In view of Proposition~\ref{prop:Im(phi)Step1}, (\romannumeral3) of the corollary is immediate from the second assertion in the theorem. For~(\romannumeral2) we only have to note that $\Image(\phi)$ contains the commutator subgroup of~$K_0$ and that $K_{0,\ell} = \cG(\bbZ_\ell)$ for almost all~$\ell$.

It remains to deduce~(\romannumeral1) from the theorem. With the notation of~\ref{ssec:rechom}, $q^{-1}\bigl(\Image(\rec)\bigr)$ is a normal subgroup of~$G(\Af)$, with profinite abelian quotient $G(\Af)/q^{-1}\bigl(\Image(\rec)\bigr) \cong \Coker(\rec)$. If $m$ is the exponent of $\Coker(\rec)$, it follows that $K_{0,\ell}/\Image(\phi_\ell)$ is a compact abelian $\ell$-adic analytic group that is killed by~$m$. Hence it is finite. It follows that $\Image(\phi_\ell)$ has finite index in~$\cG(\bbZ_\ell)$ for all~$\ell$, and so it suffices to prove~(\romannumeral1) for all~$\ell$ sufficiently large. 

Choose a multiple~$M$ of~$m$ such that $\cG$ is reductive over~$\bbZ[1/M]$ and $K_{0,\ell} = \cG(\bbZ_\ell)$ for all $\ell > M$. Fix an $\ell > M$, and let $\cG_0 = \cG \otimes \bbF_\ell$ denote the characteristic~$\ell$ fibre of~$\cG$. Let $\tilde{\cG}_0 \to \cG_0^\der$ be the simply connected cover of the derived subgroup. The image of $\tilde{\cG}_0(\bbF_\ell) \to \cG_0^\der(\bbF_\ell)$ is the normal subgroup $\cG_0(\bbF_\ell)^+ \triangleleft\, \cG_0(\bbF_\ell)$ that is generated by the $\ell$-Sylow subgroups. 

Still with $\ell > M$, the image of~$\phi_\ell$ contains the subgroup of~$\cG(\bbZ_\ell)$ generated by the $\ell$-Sylow groups, and hence contains all elements $g \in \cG(\bbZ_\ell)$ whose reduction modulo~$\ell$ lies in $\cG_0(\bbF_\ell)^+$. It therefore suffices to bound the $m$-torsion in $\cG_0(\bbF_\ell)/\cG_0(\bbF_\ell)^+$ by a constant independent of~$\ell$. As a first step, let $\mu$ be the kernel of $\tilde{\cG}_0 \to \cG_0^\der$; this is a group of multiplicative type whose rank~$|\mu|$ only depends on~$G$. (If $R$ is the absolute rank of~$G$, it is known that $\mu$ has rank at most~$2^R$.) The quotient $\cG_0^\der(\bbF_\ell)/\cG_0(\bbF_\ell)^+$ injects into $H^1(\bbF_\ell,\mu)$, which is a quotient of~$\mu(\overline{\bbF}_\ell)$ (cohomology of procyclic groups). In particular, $\bigl[\cG_0^\der(\bbF_\ell):\cG_0(\bbF_\ell)^+\bigr]$ divides~$|\mu|$. It therefore suffices to bound the $m$-torsion in $\cG_0(\bbF_\ell)/\cG_0^\der(\bbF_\ell)$ by a constant independent of~$\ell$. Writing $\cG_0^\ab = \cG_0/\cG_0^\der$, the group $\cG_0(\bbF_\ell)/\cG_0^\der(\bbF_\ell)$ is a subgroup of~$\cG_0^\ab(\bbF_\ell)$. Further, if $r$ is the rank of the torus~$G^\ab$, the kernel of multiplication by~$m$ on~$\cG_0^\ab$ is a finite group scheme of rank~$r^m$; hence the $m$-torsion in $\cG_0^\ab(\bbF_\ell)$ has cardinality at most~$r^m$.
\end{proof}

We now turn to the proof of Theorem~\ref{thm:Coker(rec)}. In \ref{ssec:PfTorus} and~\ref{ssec:Gdersc} we first prove the result in two special cases; the general case is then deduced in~\ref{ssec:EndPf}. To prove that $\Coker(\rec)$ has finite exponent (resp., is finite), we use Deligne's group-theoretic description of the reciprocity map. (See \cite{DelShimura}, Th\'eor\`eme~2.6.3.) If $E = E(G,X)$ is the reflex field, we simply write~$E^*$ for the $\bbQ$-torus $\Res_{E/\bbQ}\, \bbG_{\mult,E}$. Class field theory gives an isomorphism $\pi_0\pi(E^*) \isomarrow \Gal(E^\ab/E)$, which we normalize as in \cite{DelShimura},~0.8.

We start with a general remark that is useful to us.

\begin{remark}
\label{rem:PlatRapin}
Let $f \colon G_1 \to G_2$ be a surjective homomorphism of reductive $\bbQ$-groups. Factor~$f$ as 
\[
G_1 \twoheadrightarrow G_2^\prime \xrightarrow{\psi} G_2\, ,
\] 
where $G_2^\prime = G_1/\Ker(f)^0$. We make the simplifying assumption that $\Ker(\psi)$ (which is the finite \'etale group scheme $\pi_0\bigl(\Ker(f)\bigr)$) is commutative, as this is the only case we need.

By \cite{PlatonRapin}, Proposition~6.5, $G_1(\bbA) \to G_2^\prime(\bbA)$ has open image. The commutativity of~$\Ker(\psi)$ implies that the image of $\psi(\bbA) \colon G_2^\prime(\bbA) \to G_2(\bbA)$ is a normal subgroup.  Further, if $M$ is the rank of $\Ker(\psi)$ then $H^1\bigl(\Spec(\bbA)_\text{\'et},\Ker(\psi)\bigr)$, and hence also the cokernel of~$\psi(\bbA)$, is killed by~$M$.
\end{remark}

\subsection{The case of a torus.} 
\label{ssec:PfTorus}
Suppose $G = T$ is a torus, in which case $X = \{h\}$ is a singleton. By definition of the reflex field~$E$, the corresponding cocharacter $\mu\colon \bbG_{\mult,\bbC} \to T_\bbC$ is defined over~$E$; so we have a homomorphism $\mu\colon \bbG_{\mult,E} \to T_E$. Restricting scalars to~$\bbQ$ and composing with the norm map then gives a homomorphism of algebraic tori
\[
\nu \colon E^* \xrightarrow{\Res(\mu)} \Res_{E/\bbQ} T_E \xrightarrow{\Norm_{E/\bbQ}} T\, .
\]
Via the class field isomorphism, the reciprocity map~\eqref{eq:rec} is the inverse of the composition
\[
\pi_0\pi(E^*) \xrightarrow{\,\pi_0\pi(\nu)\,} \pi_0\pi(T) \xrightarrow{~\pr~} \bar\pi_0\pi(T)\, .
\]

We apply what was explained in~\ref{ssec:pi1Borov}, taking $M=T$. In this case $\pi_1(T)$ is just the cocharacter group $X_*(T)$. With $\Gamma=\Gal(\Qbar/\bbQ)$, the assumption that $T$ is the Mumford-Tate group of~$h$ means that $\bbZ[\Gamma] \cdot \mu$ has finite index in~$X_*(T)$. Further, $(T,\{h\})$ is maximal if and only if $X_*(T)$ is generated by~$\mu$ as a $\bbZ[\Gamma]$-module. 

On the other hand, by definition of the reflex field~$E$, the stabilizer of $\mu \in X_*(T)$ in~$\Gamma$ is precisely the subgroup $\Gamma_E = \Gal(\Qbar/E) \subset \Gamma$. If $\mathfrak{a} \subset \bbZ[\Gamma]$ is the left ideal generated by the augmentation ideal of~$\bbZ[\Gamma_E]$, we have $X_*(E^*) \cong \bbZ[\Gamma]/\mathfrak{a}$ as Galois modules, and $X_*(\nu) \colon X_*(E^*) \to X_*(T)$ is the map given by $(\gamma \bmod \mathfrak{a}) \mapsto \gamma \cdot \mu$. {}From the preceding remarks it therefore follows that $\nu$, as a homomorphism of algebraic tori, is surjective, and that $\Ker(\nu)$ is connected if $(T,\{h\})$ is Hodge-maximal. The assertion of~\ref{thm:Coker(rec)} now follows from Remark~\ref{rem:PlatRapin}.

\subsection{}
\label{ssec:recFunct}
Let $f\colon (G_1,X_1) \to (G_2,X_2)$ be a morphism of Shimura data. Let $E_i$ ($i=1,2$) be the reflex field of $(G_i,X_i)$, and denote by $E_i^\ab$ its maximal abelian extension. Then $E_1$ is a finite extension of~$E_2$ in~$\bbC$ and we have a commutative diagram
\begin{equation}
\label{eq:G1X1->G2X2}
\begin{tikzpicture}[baseline=(current bounding box.center)]
\node (A) at (0,1.8) {$\Gal(E_1^\ab/E_1)$};
\node (B) at (5.5,1.8) {$\bar\pi_0\pi(G_1) \cong  G_1(\Af)/G_1(\bbQ)_+^-$};
\node (C) at (0,0) {$\Gal(E_2^\ab/E_2)$};
\node (D) at (5.5,0) {$\bar\pi_0\pi(G_2) \cong G_2(\Af)/G_2(\bbQ)_+^-$};
\node (Bshift) at (6.2,1.8) {$\phantom{\bar\pi_0\pi(G_1) \cong  G_1(\Af)/G_1(\bbQ)_+^-}$};
\node (Dshift) at (6.2,0) {$\phantom{\bar\pi_0\pi(G_2) \cong G_2(\Af)/G_2(\bbQ)_+^-}$};
\draw[->] (A) -- (C) node[pos=.5,left] {};
\draw[->] (A) -- (B) node[pos=.5,above] {$\scriptstyle \rec_{(G_1,X_1)}$};
\draw[->] (C) -- (D) node[pos=.5,above] {$\scriptstyle \rec_{(G_2,X_2)}$};
\draw[->] (Bshift) -- (Dshift) node[pos=.48,right] {$\scriptstyle \bar{f}$};
\end{tikzpicture}
\end{equation}
in which $\bar{f}$ denotes the map induced by~$f$. The image of the left vertical map is a subgroup of finite index in $\Gal(E_2^\ab/E_2)$.

\subsection{The case when $G^\der$ is simply connected.} 
\label{ssec:Gdersc}
Next we treat the case when the derived group~$G^\der$ is simply connected. Let $G^\ab = G/G^\der$ and let $p \colon G \to G^\ab$ be the canonical map. Then $h^\ab = p \circ h$ is independent of $h \in X$, and $(G^\ab,\{h^\ab\})$ is a Shimura datum. By Remark~\ref{rem:MorphShim}(\romannumeral2), if $(G,X)$ is maximal, so is $(G^\ab,X^\ab)$. We apply~\ref{ssec:recFunct} to the morphism $p \colon (G,X) \to (G^\ab,\{h^\ab\})$. By \cite{DelTravShim}, Th\'eor\`eme~2.4, the right vertical map in the diagram is surjective with finite kernel. (Deligne's result says that $p$ induces an isomorphism on the groups that we denote by~$\pi_0\pi$; the groups $\bar\pi_0\pi$ are quotients of these by finite subgroups.) The theorem for $(G,X)$ therefore follows from the result for $(G^\ab,X^\ab)$, which was proven in~\ref{ssec:PfTorus}.

\begin{lemma}
\label{lem:CoveringGX}
Let $(G,X)$ be a Shimura datum as in~\textup{\ref{ssec:ShimRepDef}}. Then there exists a Shimura datum $(\tilde{G},\tilde{X})$ and a morphism $f \colon (\tilde{G},\tilde{X}) \to (G,X)$ such that
\begin{enumerate}[label=\textup{(\alph*)}]
\item the group $\tilde{G}$ is the generic Mumford-Tate group on~$\tilde{X}$;
\item the homomorphism $f \colon \tilde{G} \to G$ is surjective, and the induced $f^\der\colon \tilde{G}^\der \to G^\der$ is the simply connected cover of~$G^\der$.
\end{enumerate}
Moreover, if $(G,X)$ is maximal, we can choose $(\tilde{G},\tilde{X})$ such that it is maximal, too, and such that the kernel of $f \colon \tilde{G} \to G$ is connected.
\end{lemma}

\begin{proof}
By \cite{MilShih}, Application~3.4, there exists a morphism of Shimura data $f_1 \colon (G_1,\tilde{X}) \rightarrow (G,X)$ such that $f_1 \colon G_1 \to G$ is surjective, $\Ker(f_1)$ is a torus, and $f_1^\der \colon G_1^\der \to G^\der$ is the simply connected cover. Let $\tilde{G} \subset G_1$ be the generic Mumford-Tate group on~$\tilde{X}$, and let $f\colon \tilde{G}\to G$ be the restriction of~$f_1$ to~$\tilde{G}$. The assumption that $G$ is the generic Mumford-Tate group on~$X$ implies that $f$ is surjective, and as $\tilde{G}$ is normal in~$G_1$, this implies that $\tilde{G}^\der = G_1^\der$. So $(\tilde{G},\tilde{X})$ and~$f$ satisfy (a) and~(b).

Next assume $(G,X)$ is maximal. We claim that for any $f \colon (\tilde{G},\tilde{X}) \to (G,X)$ such that (a) and~(b) hold, $\Ker(f)$ is connected (hence a torus). Indeed, let $(G_2,X_2)$ be the quotient of $(\tilde{G},\tilde{X})$ by $\Ker(f)^0$. We then have an induced morphism of Shimura data $(G_2,X_2) \to (G,X)$. The map $G_2 \to G$ is an isogeny with kernel the group scheme $\pi_0\bigl(\Ker(f)\bigr)$ of connected components of $\Ker(f)$. By maximality of $(G,X)$ this implies that $\pi_0\bigl(\Ker(f)\bigr)$ is trivial, which proves the claim.

The Shimura datum $(\tilde{G},\tilde{X})$ we have obtained need not be maximal. By Remark~\ref{rem:MorphShim}(\romannumeral3), there exists an isogeny $(\hat{G},\hat{X}) \to (\tilde{G},\tilde{X})$ with $(\hat{G},\hat{X})$ maximal, and by \ref{rem:MorphShim}(\romannumeral1) $\hat{G}$ is the generic Mumford-Tate group on~$\hat{X}$. As $\hat{G} \to \tilde{G}$ is an isogeny and $\tilde{G}^\der$ is simply connected, we have $\hat{G}^\der \isomarrow \tilde{G}^\der$, and as we have already seen, the maximality of $(G,X)$ implies that the kernel of $\hat{G} \to G$ is connected. 
\end{proof}

\subsection{}
\label{ssec:EndPf}
To complete the proof of Theorem~\ref{thm:Coker(rec)}, let $(G,X)$ be the Shimura datum considered in the assertion, and take $f\colon (\tilde{G},\tilde{X}) \to (G,X)$ as in Lemma~\ref{lem:CoveringGX}. As shown in~\ref{ssec:Gdersc}, the theorem is true for $(\tilde{G},\tilde{X})$. In view of diagram~\eqref{eq:G1X1->G2X2}, it suffices to show that the induced map $\tilde{G}(\Af) \to G(\Af)$ has cokernel of finite exponent, and that the image of this map is open if $(G,X)$ is maximal. This follows from Remark~\ref{rem:PlatRapin}.

%% file: Sect4-GalGen
\section{Galois-generic points}
\label{sec:GalGen}

\subsection{}
\label{ssec:Galsect}
Let $S$ be a geometrically connected scheme of finite type over a field~$k$. Assume given a continuous representation $\psi \colon \pi_1(S) \to G(\Af)$, where $G$ is an algebraic group over~$\bbQ$. (Throughout, we omit the choice of a geometric base point of~$S$ from the notation.) If $\ell$ is a prime number, we denote by $\psi_\ell \colon \pi_1(S) \to G(\bbQ_\ell)$ the $\ell$-primary component of~$\psi$.

If $y$ is a point of~$S$, we have a homomorphism $\sigma_y\colon \pi_1(y) \to \pi_1(S)$, well-determined up to conjugation. (Recall that $\pi_1(y)$ is the absolute Galois group of the residue field~$k(y)$.)

\begin{definition}
(\romannumeral1) A point $y \in S$ is said to be \emph{Galois-generic} with respect to~$\psi$ if the image of $\psi \circ \sigma_y$ is open in the image of~$\psi$.

(\romannumeral2) A point $y \in S$ is said to be \emph{$\ell$-Galois-generic} with respect to~$\psi$ if the image of $\psi_\ell \circ \sigma_y$ is open in the image of~$\psi_\ell$.
\end{definition}

If it is clear which $\psi$ we mean, we omit the phrase ``with respect to~$\psi$''.

Clearly, if a point~$y$ is Galois-generic, it is $\ell$-Galois generic for all~$\ell$. The following theorem by the first author and Kret (see~\cite{CadKret}, Theorem~1.1) shows that for the representation associated with a Shimura variety of abelian type, the converse holds.

\begin{theorem}
\label{thm:CadKret}
Let $(G,X)$ be a Shimura datum of abelian type, $K_0 \subset G(\Af)$ a neat compact open subgroup, $h_0\in X$ a base point. Let $\phi_{\bar{s}} \colon \pi_1(S_0) \to K_0$ be the associated representation, as defined in~\textup{\ref{ssec:ShimRepDef}}. If a point $y \in S$ is $\ell$-Galois-generic for some prime number~$\ell$ then $y$ is Galois-generic.
\end{theorem}

%% file: Sect5-AV
\section{Application to abelian varieties}
\label{sec:AV}

\subsection{}
\label{ssec:Ag}
For $g \geq 1$, equip $\bbZ^{2g}$ with the standard symplectic form, and let $\CSp_{2g}$ be the reductive group over~$\bbZ$ of symplectic similitudes of~$\bbZ^{2g}$. Let $\mathfrak{H}_g^\pm$ be the set of homomorphisms $h \colon \bbS \to \CSp_{2g,\bbR}$ that define a Hodge structure of type $(-1,0)+(0,-1)$ on~$\bbZ^{2g}$ for which $\pm 2\pi i \cdot \psi$ is a polarization. The real group $\CSp_{2g}(\bbR)$ acts transitively on~$\mathfrak{H}_g^\pm$, and the pair $(\CSp_{2g},\mathfrak{H}_g^\pm)$ is a Shimura datum with reflex field~$\bbQ$.

Let $K(3) \subset \CSp_{2g}(\Zhat)$ be the subgroup of elements that reduce to the identity modulo~$3$. Then $\Sh_{K(3)}(\CSp_{2g},\mathfrak{H}_g^\pm)$ is isomorphic (over~$\bbQ$) to the moduli space~$\cA_{g,3}$ of $g$-dimensional principally polarized abelian varieties with a (Jacobi) level~$3$ structure. See for instance \cite{DelTravShim}, Section~4. In what follows we identify the two schemes.

Let $(B,\lambda)$ be a principally polarized abelian variety of dimension~$g$ over a subfield $k \subset \bbC$ that is finitely generated over~$\bbQ$. Assume that all $3$-torsion points of~$B$ are $k$-rational. (This implies that $\bbQ(\zeta_3) \subset k$.) Choose a similitude $i\colon H_1\bigl(B(\bbC),\bbZ\bigr) \isomarrow \bbZ^{2g}$. This gives a Hodge structure on~$\bbZ^{2g}$; let $h_0 \in \mathfrak{H}_g^\pm$ be the corresponding point. As in section~\ref{ssec:ShimRepDef}, $[h_0,eK(3)]$ defines a $\bbC$-valued point $\bar{t}_0 \in \cA_{g,3}(\bbC)$. The corresponding level~$3$ structure on $(B_\bbC,\lambda)$ is defined over~$k$, which means that $\bar{t}_0$ comes from a $k$-valued point $t_0 \in \cA_{g,3}(k)$ by composing it with the given embedding $k \hookrightarrow \bbC$.

Let $\cA_0 \subset \cA_{g,3} \otimes \bbQ(\zeta_3)$ be the irreducible component such that $t_0 \in \cA_0(k)$. This component is geometrically irreducible. The construction of~\ref{ssec:ShimRepDef} gives a representation
\[
\phi_{\bar{t}} \colon \pi_1(\cA_0,\bar{t}_0) \to K(3) \subset \GL_{2g}(\Zhat)\, .
\]
(Here, as in~\ref{ssec:ShimRepDef}, $\bar{t}$ refers to the compatible system of base points $(\bar{t}_K)$ obtained from~$h_0$.) The point~$t_0$ gives rise to a homomorphism $\sigma_{t_0} \colon \Gal(\kbar/k) \to \pi_1(\cA_0,\bar{t}_0)$ such that the composition with the projection $\pi_1(\cA_0,\bar{t}_0) \to \Gal\bigl(\Qbar/\bbQ(\zeta_3)\bigr)$ is the natural homomorphism $\Gal(\kbar/k) \to \Gal\bigl(\Qbar/\bbQ(\zeta_3)\bigr)$. Note that $\sigma_{t_0}$ is canonically defined, not only up to conjugation.

On the other hand, if we let $H = H_1\bigl(B(\bbC),\bbZ\bigr)$ then we may identify the full Tate module $\hat{H} = \varprojlim_n\, B[n]\bigl(\kbar\bigr)$ (limit over all positive integers~$n$, partially ordered by divisibility) with $H \otimes \Zhat$. Via the chosen similitude~$i$ we obtain an isomorphism $\hat{H} \isomarrow \Zhat^{2g}$. The natural Galois action on~$\hat{H}$ therefore gives a representation
\[
\rho_B \colon \Gal(\kbar/k) \to \GL_{2g}\bigl(\Zhat\bigr)\, .
\] 

The following result is an immediate consequence of the definitions and the modular interpretation of $\Sh(\CSp_{2g},\mathfrak{H}_g^\pm)$. See also \cite{UllYafQuebec}, Remark~2.8, as well as the next section, where we discuss the analogous (but slightly more involved) case of K3 surfaces.

\begin{proposition}
\label{prop:GalCompat}
The representations $\phi_{\bar{t}} \circ \sigma_{t_0}$ and~$\rho_B$ are the same.
\end{proposition}

The main result of this section is that the usual form of the Mumford-Tate conjecture implies the integral and adelic versions of the Mumford-Tate conjecture as formulated in~\ref{conj:IMTC} and~\ref{conj:AMTC}.

\begin{theorem}
\label{thm:MainThmAV}
Let $B$ be an abelian variety over a subfield $k \subset \bbC$ that is finitely generated over~$\bbQ$. Assume that for some prime number~$\ell$ the Mumford-Tate conjecture for~$B$ is true. Then the integral and adelic Mumford-Tate conjectures for~$B$ are true as well.
\end{theorem}

Note that in this case the last part of Conjecture~\ref{conj:IMTC} says that the image of~$\rho_B$ contains an open subgroup of $\Zhat^* \cdot \id$. This is in fact a result proven by Wintenberger~\cite{WintenbLang}. (The result is stated in loc.\ cit.\ only for $k$ a number field; this implies the same result over finitely generated fields, as we can specialize~$B$ to an abelian variety over a number field in such a way that the Mumford-Tate group does not change.)

\begin{proof}
In proving the theorem, we may replace the ground field~$k$ with a finite extension and $B$ with an isogenous abelian variety. Hence we may assume that $B$ admits a principal polarization and that all $3$-torsion points of~$B$ are $k$-rational. This puts us in the situation of~\ref{ssec:Ag}. We retain the notation and choices introduced there. 

Via the chosen similitude $H \isomarrow \bbZ^{2g}$ we may view the Mumford-Tate group~$G_\Betti$ of~$B$ as an algebraic subgroup of~$\CSp_{2g,\bbQ}$. We take its Zariski closure $G_\Betti \subset \CSp_{2g}$ as integral model. (There is no need to introduce new notation for this integral form.) With $h_0 \in \mathfrak{H}_g^\pm$ as in~\ref{ssec:Ag}, let $X \subset \mathfrak{H}_g^\pm$ be the $G_\Betti(\bbR)$-orbit of~$h_0$. The pair $(G_\Betti,X)$ is a Shimura datum, and by construction we have a morphism $f \colon (G_\Betti,X) \to (\CSp_{2g},\mathfrak{H}_g^\pm)$. Let $K_0 = f^{-1}\bigl(K(3)\bigr)$, which is a neat compact open subgroup of~$G_\Betti(\Af)$, and, with $E$ the reflex field of $(G_\Betti,X)$, let $\Sh(f) \colon \Sh_{K_0}(G_\Betti,X) \to \cA_{g,3} \otimes E$ be the morphism induced by~$f$.

Let $\bar{s}_0 = [h_0,eK_0]$, which is a $\bbC$-valued point of $\Sh_{K_0}(G_\Betti,X)$ whose image under~$\Sh(f)$ is~$\bar{t}_0$. As in~\ref{ssec:ShimRepDef}, let $S_{0,\bbC} \subset \Sh_{K_0}(G_\Betti,X)_\bbC$ be the irreducible component containing~$\bar{s}_0$, let $F \subset \bbC$ be its field of definition, and let $S_0 \subset \Sh_{K_0}$ be the geometrically irreducible component thus obtained. Possibly after replacing~$k$ with a finite extension, $\bar{s}_0$ comes from a point $s_0 \in S_0(k)$, whose image in $\cA_{g,3}(k)$ is the point $t_0$ that corresponds to $(B,\lambda)$ equipped with a suitable level~$3$ structure. This point~$s_0$ gives rise to a homomorphism $\sigma_{s_0} \colon \Gal(\kbar/k) \to \pi_1(S_0,\bar{s}_0)$. With $\phi_{\bar{s}} \colon \pi_1(S_0,\bar{s}_0) \to K_0 \subset G_\Betti(\Zhat)$ the homomorphism~\eqref{eq:phi}, it follows from the functoriality explained in Remark~\ref{rem:phi} together with Proposition~\ref{prop:GalCompat} that $\phi_{\bar{s}} \circ \sigma_{s_0}$ is the Galois representation on the (full) Tate module of~$B$. 

Let $\ell$ be a prime number and $\phi_\ell \colon \pi_1(S_0,\bar{s}_0) \to  G_\Betti(\bbZ_\ell)$ the $\ell$-primary component of~$\phi$. If the Mumford-Tate conjecture for $B$ is true at~$\ell$, $s_0$ is $\ell$-Galois-generic with respect to the representation~$\phi$. (See the remark after~\ref{conj:IMTC}.) By Theorem~\ref{thm:CadKret}, $s_0$ is then Galois-generic, and in view of the description of~$\Image(\phi)$ given in Corollary~\ref{cor:Im(phi)}, this implies that Conjecture~\ref{conj:IMTC} and, if the Hodge structure~$H$ is maximal, Conjecture~\ref{conj:AMTC} are true for~$B$.
\end{proof}

\begin{remarks}
(\romannumeral1) There are many special classes of abelian varieties for which the Mumford-Tate conjecture is known. For a sample of such results, see for instance \cite{Pink}, Section~5, or the more recent~\cite{Lombardo} and the references contained therein. On the other hand, already for abelian varieties of dimension~$4$ the Mumford-Tate conjecture remains open.

(\romannumeral2) When we restrict our attention to the image of~$\rho$ intersected with the derived subgroup of~$G_\Betti$, a result related to our Theorem~\ref{thm:MainThmAV} was obtained by Hui and Larsen; see~\cite{HuiLars}, Theorem~4.2. Their work is based on a very different approach.
\end{remarks}

\begin{example}
As a first application we recover the result given in~\cite{SerreMotives}, 11.11. If $B$ is a $g$-dimensional abelian variety with $g$~odd (or $g=2$, or $g=4$) and $\End(B_\kbar) = \bbZ$, it is known that the Mumford-Tate conjecture for~$B$ is true and that the Mumford-Tate group is the full $\CSp_{2g}$. (See~\cite{Pink}, Theorem~5.14, for a more general result.) It is easily seen that the Shimura datum $(\CSp_{2g},\mathfrak{H}_g^\pm)$ is maximal; the conclusion therefore is that in this case the image of the representation~$\rho_B$ is open in $\CSp_{2g}(\Af)$.
\end{example}

Next we want to give examples of abelian varieties~$B$ over finitely generated subfields of~$\bbC$ for which the Hodge structure $H_1\bigl(B(\bbC),\bbQ\bigr)$ is not Hodge-maximal. Such examples of course also give us Shimura data of Hodge type that are not maximal. The first examples we discuss are of CM type; after that, we discuss an example in which the Mumford-Tate group is the almost direct product of the homotheties and a semisimple group.

\begin{example}
For our first construction, we start with a totally real field~$E_0$ of degree~$g$ over~$\bbQ$. Let $\sigma_1, \ldots, \sigma_g$ be the complex embeddings of~$E_0$. Let $k$ be an imaginary quadratic field. Then $E = k \cdot E_0$ is a CM field. Fix an embedding $\alpha \colon k \to \bbC$, and let $\tau_i$ ($i=1,\ldots,g$) be the complex embedding of~$E$ that extends~$\sigma_i$ and such that $\tau_i|_k = \alpha$. Thus, $T = \bigl\{\tau_1,\ldots,\tau_g,\bar\tau_1,\ldots,\bar\tau_g\bigr\}$ is the set of complex embeddings of~$E$. 

Consider the CM type $\Phi$ on~$E$ given by
\[
\Phi = \{\tau_1,\bar\tau_2,\ldots,\bar\tau_g\}\, .
\]
The pair $(E,\Phi)$ gives rise to an isogeny class of $g$-dimensional complex abelian varieties~$B$, determined by the rule that $H_1\bigl(B(\bbC),\bbQ\bigr) \cong E$ as a $\bbQ$-vector space, with Hodge decomposition of $H_1\bigl(B(\bbC),\bbC\bigr) \cong \oplus_{\tau \in T}\; \bbC^{(\tau)}$ given by the rule that $\bbC^{(\tau)}$ is of type $(-1,0)$ if $\tau \in \Phi$ and of type $(0,-1)$ otherwise. If $g > 1$ then $\Phi$ is a primitive CM type; in this case $B$ is simple. As any abelian variety of CM type, $B$ is defined over a number field, and by a result of Pohlmann~\cite{Pohlmann} the Mumford-Tate conjecture is true for~$B$.

As in Section~\ref{sec:AdelicRep}, if $F$ is a number field we simply write~$F^*$ for the torus $\Res_{F/\bbQ}\, \bbG_{\mult,F}$. Let $\Norm\colon E^* \to E_0^*$ be the norm homomorphism, and let $U \subset E^*$ be the subtorus given by $U = \Norm^{-1}(\bbQ^*)$. The cocharacter group $X_*(E^*)$ is the free $\bbZ$-module on the set~$T$. The cocharacter group of~$U$ is given by
\[
X_*(U) = \left\{\sum_{i=1}^g a_i \tau_i + \sum_{i=1}^g b_i \bar\tau_i \in X_*(E^*) \Biggm| \text{$a_i +b_i$ is independent of~$i$}\right\}\, .
\]
The elements $f_i = \tau_i - \bar\tau_i$ ($i=1,\ldots,g$) together with $f_{g+1} = \sum_{i=1}^g\, \bar\tau_i$ form a basis for~$X_*(U)$.

The cocharacter $\mu\colon \bbG_{\mult,\bbC} \to E^*$ corresponding to the Hodge structure $H_1\bigl(B(\bbC),\bbQ\bigr)$ is given by $\mu = \tau_1 + \bar\tau_2 + \cdots + \bar\tau_g = f_1 + f_{g+1}$. The Galois conjugates of~$\mu$ are the elements $f_i + f_{g+1}$ for $i=1,\ldots,g$ together with their complex conjugates $f_1 + \cdots + \hat{f}_i + \cdots + f_g + f_{g+1}$, for $i=1,\ldots,g$. These are cocharacters in $X_*(U)$, and for $g > 2$ they span a submodule of index $g-2$ in~$X_*(U)$. The conclusion, therefore, is that $U$ is the Mumford-Tate group of~$B$ if $g>2$, and that $H_1\bigl(B(\bbC),\bbQ\bigr)$ is not Hodge-maximal if $g>3$. In this last case, the image of the adelic Galois representation is therefore not open in the adelic points of the Mumford-Tate group.
\end{example}

\begin{example}
For our final example, we consider a Shimura datum $(G,X)$ such that
\begin{enumerate}[label=\textup{(\alph*)}]
\item $G$ is an inner form of a split group; 
\item $\pi_1(G)$ is non-cyclic.
\end{enumerate}

\noindent
Note that (b) holds if $G^\ab$ has dimension at least~$2$, or if $\dim(G^\ab) = 1$ and $G^\der$ is not simply-connected. (In the latter case this follows using \cite{Borovoi}, Cor.~1.7.) If (a) holds, $\Gal(\Qbar/\bbQ)$ acts trivially on~$\pi_1(G)$ (see~\cite{Borovoi}, Lemma~1.8), and by what was explained in \ref{ssec:pi1Borov} and~\ref{ssec:HmaxGX}, we conclude that any $(G,X)$ satisfying (a) and~(b) is non-maximal.

To obtain a concrete example, let $D$ be a quaternion algebra over~$\bbQ$ that is non-split at infinity, i.e., $D \otimes_\bbQ \bbR$ is Hamilton's quaternion algebra~$\bbH$. The canonical involution~$*$ on~$D$ is then a positive involution. Let $n = 2r$ be an even positive integer with $n \geq 6$, and consider a free (left-) $D$-module~$V$ of rank~$n$ equipped with a $(-1)$-hermitian form~$\Psi$ of discriminant~$1$ and Witt index~$r$. Let $G^\prime = \UU_D(V,\Psi)$ be the corresponding unitary group, which we view as an algebraic subgroup of~$\GL_\bbQ(V)$, and let $G \subset \GL_\bbQ(V)$ be the algebraic group generated by~$G^\prime$ together with the homotheties $\bbG_\mult \cdot \id$. The group $G^\prime \otimes \bbR$ is isomorphic to the identity component of~$\UU_n^*(\bbH)$ (which in some literature is denoted by~$\SO^*(2n)$), and there is a unique $G(\bbR)$-conjugacy class of homomorphisms $h\colon \bbS \to G_\bbR$ that make the pair $(G,X)$ a Shimura datum of PEL type. (Cf.\ \cite{DelShimura}, Section~1.3.) The corresponding Shimura variety parametrizes polarized abelian varieties of dimension~$2n$ with an action by (an order in)~$D$, which is of Albert Type~III. As $G^\prime$ is a $\bbQ$-simple group, $G$ is the generic Mumford-Tate group on~$X$.

Our assumption that $\Psi$ has trivial discriminant implies that the index of~$G^\prime$ is~${}^1\textrm{D}_n$; see \cite{Tits}, Table~II, pages 56--57. This means that $G^\prime$ (and hence also~$G$) is an inner form of the split form, i.e., condition~(a) is satisfied. On the other hand, $\dim(G^\ab) = 1$ and $G^\der$ is not simply-connected, so also (b) is satisfied. (In fact, $\pi_1(G) \cong \bbZ \times (\bbZ/2\bbZ)$.) We conclude that $(G,X)$ is not maximal. If $B$ is a complex abelian variety that corresponds to a Hodge-generic point of the Shimura variety defined by $(G,X)$, the Hodge structure $H_1\bigl(B(\bbC),\bbQ\bigr)$ is not Hodge-maximal.
\end{example}

%% file: Sect6-K3
\section{Application to K3 surfaces}
\label{sec:K3}

\subsection{}
\label{ssec:pi1Res}
Let $L$ be a number field, $G$ a connected reductive group over~$L$, and let $M = \Res_{L/\bbQ}\, G$. Then $M_\bbC \cong \prod_{\sigma \in \Sigma}\, G_\sigma$, where $\Sigma$ is the set of complex embeddings of~$L$ and $G_\sigma = G \otimes_{L,\sigma} \bbC$. With $\Gamma = \Gal(\Qbar/\bbQ)$ and $\Gamma_L = \Gal(\Qbar/L)$, the fundamental group~$\pi_1(M)$ is the $\Gamma$-module obtained from the $\Gamma_L$-module~$\pi_1(G)$ by induction.

Let $\mu$ be a complex cocharacter of~$M$. As in~\ref{ssec:pi1Borov}, its conjugacy class~$\cC$ defines an element $[\cC]\in \pi_1(M)$. Let $W \subset \pi_1(M)$ be the $\bbZ[\Gamma]$-submodule generated by~$[\cC]$.

Suppose there is a unique $\tau \in \Sigma$ such that the projection~$\mu_\tau$ of~$\mu$ onto the factor~$G_\tau$ is non-trivial. View~$L$ as a subfield of~$\bbC$ via~$\tau$; then $\mu_\tau$ is a complex cocharacter of~$G$. Its conjugacy class~$\cC_\tau$ defines an element $[\cC_\tau]$ in~$\pi_1(G)$. Let $W_\tau \subset \pi_1(G)$ be the $\bbZ[\Gamma_L]$-submodule that it generates. In this situation we have $W = \bbZ[\Gamma] \otimes_{\bbZ[\Gamma_L]} W_\tau$ as submodule of $\pi_1(M) = \bbZ[\Gamma] \otimes_{\bbZ[\Gamma_L]} \pi_1(G)$. Consequently, $\mu$ is maximal as a cocharacter of~$M$ if and only if $\mu_\tau$ is maximal as a complex cocharacter of~$G$.

\begin{proposition}
\label{prop:K3Maximal}
Let $V$ be a polarizable $\bbQ$-Hodge structure of K3 type, by which we mean that $V$ is of type $(-1,1) + (0,0) + (1,-1)$ with Hodge numbers $1$--$n$--$1$ for some~$n$. Then $V$ is Hodge-maximal.
\end{proposition}

\begin{proof}
Without loss of generality we may assume that $V$ is simple as a Hodge structure, which in this case means that there are no non-zero Hodge classes in~$V$. (By definition, Hodge-maximality only depends on the Mumford-Tate group~$M$ and the defining homomorphism $h \colon \bbS \to M_\bbR$; these do not change if we replace $V$ with $V \oplus \bbQ(0)$.) Let $L = \End_\QHS(V)$ be the endomorphism algebra of~$V$ as a $\bbQ$-Hodge structure, and choose a polarization form $\psi \colon V \times V \to \bbQ$. As shown by Zarhin in~\cite{ZarhinK3}, $L$ is a field which is either totally real or a CM field. 

First suppose $L$ is totally real. By \cite{vGeemen}, Lemma~3.2, $\dim_L(V) \geq 3$. By \cite{ZarhinK3}, Theorem~2.2.1, $M = \Res_{L/\bbQ}\, \SO_L(V,\Psi)$, where $\Psi\colon V \times V \to L$ is the unique symmetric $L$-bilinear form on~$V$ such that $\trace_{L/\bbQ} \circ \Psi = \psi$. In particular, $M$ is semisimple. If $\dim_L(V)$ is odd, $M$ is simply connected and there is nothing to prove. Next suppose $\dim_L(V) = 2l$ is even. Let $\Sigma$ be the set of complex embeddings of~$L$. Write $G = \SO_L(V,\Psi)$ and let $\mu \colon \bbG_{\mult,\bbC} \to M_\bbC \cong \prod_{\sigma \in \Sigma} G_\sigma$ be the cocharacter that gives the Hodge structure. There is a unique $\tau \in \Sigma$ such that $\mu_\tau \neq 1$, so we are in the situation of~\ref{ssec:pi1Res}. The root system of~$G_\bbC$ is of type~$\mathrm{D}_l$, and we follow the notation of~\cite{BourbLie}, Planche~\textrm{IV}. Note that the root system in this case is self-dual; further, the calculation that follows goes through without changes if $l=2$. With respect to the basis $\epsilon_1,\ldots,\epsilon_l$ for $\bbR^l = X_*(G) \otimes \bbR$, we have $X_*(G) = \bbZ^l$, and the coroot lattice $Q(R^\vee)$ consists of the vectors $(m_1,\ldots,m_l) \in \bbZ^l$ for which $\sum m_j$ is even. On the other hand, the cocharacter~$\mu_\tau$ corresponds to the vector $(1,0,\ldots,0)$; its image in $\pi_1(G) = X_*(G)/Q(R^\vee) \cong \bbZ/2\bbZ$ is therefore the non-trivial class. By what was explained in~\ref{ssec:pi1Borov} this implies the assertion.

Next suppose $L$ is a CM field. Let $L_0 \subset L$ be the totally real subfield. There is a unique symmetric hermitian form $\Psi\colon V \times V \to L$ (with respect to complex conjugation on~$L$) such that $\trace_{L/\bbQ} \circ \Psi = \psi$, and by \cite{ZarhinK3}, Theorem~2.3.1, $M = \Res_{L_0/\bbQ}\, \UU_L(V,\Psi)$. Write $G = \UU_L(V,\Psi)$, and let $\Sigma$ be the set of complex embeddings of~$L_0$. As in the totally real case there is a unique $\tau \in \Sigma$ such that the cocharacter~$\mu$ is non-trivial on the factor~$G_\tau$. If $n = \dim_L(V)$, we have $G_\bbC \cong \GL_n$ in such a way that $\mu_\tau$ is conjugate to the cocharacter $\bbG_{\mult} \to \GL_n$ given by $z \mapsto \diag(z,1,\ldots,1)$. It is straightforward to check that the corresponding class in $\pi_1(\GL_n) \cong \bbZ$ is a generator, and again by~\ref{ssec:pi1Borov} and~\ref{ssec:pi1Res} this implies the assertion.
\end{proof}

\begin{remark}
In the proposition it is essential that we work with a Hodge structure of weight~$0$. As is well-known, if $Y$ is a complex K3 surface, the Hodge structure $H = H^2_\prim\bigl(Y(\bbC),\bbQ\bigr)$ is not, in general, Hodge-maximal; but $H(1) = H \otimes \bbQ(1)$ is. For instance, if $\End_\QHS(H) = \bbQ$, the Mumford-Tate group of~$H$ is the group $\GO(H,\phi)$ of orthogonal similitudes, where $\phi$ is a polarization form. We have a non-trivial isogeny $\CSpin(H,\phi) \to \GO(H,\phi)$, such that the homomorphism $h \colon \bbS \to \GO(H,\phi)_\bbR$ that defines the Hodge structure on~$H$ lifts to a homomorphism $\tilde{h} \colon \bbS \to \CSpin(H,\phi)_\bbR$. Cf.~\cite{WintenbRelev}, 2.2.3--4. By contrast, the Mumford-Tate group of $H(1)$ is the special orthogonal group $\SO(H,\phi)$. We can still lift to $\CSpin(H,\phi)$, but the homomorphism $\CSpin(H,\phi) \to \SO(H,\phi)$ is not an isogeny.
\end{remark}

\subsection{}
\label{ssec:ModuliK3}
As a preparation for the main result of this section, we need to recall some facts about the moduli of polarized K3 surfaces. We closely follow Rizov~\cite{RizovModuli},~\cite{RizovKS}.

Fix a natural number~$d$. Let $(L_0,\psi)$ be the quadratic lattice $U^{\oplus 3} \oplus E_8^{\oplus 2}$ (with $U$ the hyperbolic lattice). With $\{e_1,f_1\}$ the standard basis of the first copy of~$U$, let $(L_{2d},\psi_{2d})$ be the sublattice $\langle e_1+df_1\rangle \oplus U^{\oplus 2} \oplus E_8^{\oplus 2}$ of~$L_0$. In what follows we write $\SO$ for the $\bbZ$-group scheme $\SO(L_{2d},\psi_{2d})$. For $n \geq 1$, let $K(n) \subset \SO(\Af)$ be the compact open subgroup of elements in~$\SO(\Zhat)$ that reduce to the identity modulo~$n$. If $K$ is an open subgroup of~$K(n)$ for some $n\geq 3$, Rizov defines in \cite{RizovModuli}, Section~6, a moduli stack $\cF_{2d,K}$ over~$\bbQ$ of K3 surfaces with a primitive polarization of degree~$2d$ and a level~$K$ structure. (In fact, Rizov does this over open parts of $\Spec(\bbZ)$, but for our purposes it suffices to work over~$\bbQ$.) By \cite{RizovKS}, Cor.~2.4.3, $\cF_{2d,K}$ is a scheme. If $(Y,\lambda)$ is a K3 surface over a field~$k$ of characteristic~$0$ equipped with a primitive polarization of degree~$2d$, a level~$K(n)$-structure on $(Y,\lambda)$ is an isometry $H^2_\prim(Y_\kbar,\bbZ/n\bbZ)\bigl(1\bigr) \isomarrow L_{2d}/nL_{2d}$. (See \cite{RizovModuli}, Example~5.1.3.) 

The construction of~\ref{ssec:ShimRepDef} has an analogue in this setting. Let $\cF_{0,\bbC}$ be an irreducible component of $\cF_{2d,K(3)} \otimes \bbC$, and let $F \subset \bbC$ be its field of definition, so that we have a geometrically irreducible component $\cF_0 \subset \cF_{2d,K(3)} \otimes F$. For $K \subset K(3)$ we have an \'etale morphism $\cF_{K,K(3)} \colon \cF_{2d,K} \to \cF_{2d,K(3)}$, which for $K$ normal in~$K(3)$ is Galois with group $K(3)/K$. Let $\cF_K \subset \cF_{2d,K} \otimes F$ be the inverse image of~$\cF_0$. Suppose we are given a compatible collection $\bar{y} = (\bar{y}_K)$ of geometric points of the~$\cF_K$, for $K$ open and normal in~$K(3)$. We write $\bar{y}_0$ for~$\bar{y}_{K(3)}$. We then have and associated homomorphism
\begin{equation}
\label{eq:PhiK3}
\Phi_{\bar{y}} \colon \pi_1(\cF_0,\bar{y}_0) \to K(3) \subset \SO(\Zhat)\, .
\end{equation}

\subsection{}
\label{ssec:K3Shim}
With $\SO$ as above, let $\Omega^\pm$ be the space of homomorphisms $h \colon \bbS \to \SO_\bbR$ that give $L_{2d} \otimes \bbQ$ a Hodge structure of type $(-1,1)+(0,0)+(1,-1)$ with Hodge numbers $1$--$19$--$1$, such that~$\pm \psi_{2d}$ is a polarization. The group $\SO(\bbR)$ acts transitively on~$\Omega^\pm$, and the pair $(\SO_\bbQ,\Omega^\pm)$ is a Shimura datum of abelian type with reflex field~$\bbQ$.

On of the main results of~\cite{RizovKS} (loc.\ cit., Thm.~3.9.1) is that for an open subgroup $K \subset K(3)$ there is an \'etale morphism of $\bbQ$-schemes
\begin{equation}
\label{eq:PeriodMap}
j_K \colon \cF_{2d,K} \to \Sh_K(\SO_\bbQ,\Omega^\pm)
\end{equation}
in such a way that for $K_2 \subset K_1$ the diagram
\begin{equation}
\label{eq:FtoSh}
\begin{tikzpicture}[baseline=(current bounding box.center)]
\node (A) at (0,1.8) {$\cF_{2d,K_2}$};
\node (B) at (4,1.8) {$\Sh_{K_2}(\SO_\bbQ,\Omega^\pm)$};
\node (C) at (0,0) {$\cF_{2d,K_1}$};
\node (D) at (4,0) {$\Sh_{K_1}(\SO_\bbQ,\Omega^\pm)$};
\draw[->] (A) -- (C) node[pos=.5,left] {$\scriptstyle \cF_{K_2,K_1}$};
\draw[->] (A) -- (B) node[pos=.5,above] {$\scriptstyle j_{K_2}$};
\draw[->] (C) -- (D) node[pos=.5,above] {$\scriptstyle j_{K_1}$};
\draw[->] (B) -- (D) node[pos=.5,right] {$\scriptstyle \Sh_{K_2,K_1}$};
\end{tikzpicture}
\end{equation}
is cartesian. The image of~$j_K$ is the complement of a divisor (ibid., 3.10(B)).

\begin{theorem}
\label{thm:K3}
Let $Y$ be a K3 surface over a subfield $k \subset \bbC$ that is finitely generated over~$\bbQ$. Let $H= H^2\bigl(Y(\bbC),\bbZ\bigr)(1)$, and let $G_\Betti \subset \GL(H)$ be the Mumford-Tate group. Let $\rho_Y \colon \Gal(\kbar/k) \to \GL(H)\bigl(\Zhat\bigr)$ be the Galois representation on $\hat{H} = H^2\bigl(Y_\kbar,\Zhat\bigr)(1)$, which we identify with $H \otimes \Zhat$ via the comparison isomorphism between singular and \'etale cohomology. Then the image of~$\rho_Y$ has a subgroup of finite index which is an open subgroup of~$G_\Betti(\Zhat)$.
\end{theorem}

\begin{proof}
The proof is very similar to that of Theorem~\ref{thm:MainThmAV}. The main difference is that for K3 surfaces the Mumford-Tate conjecture is known, due to results of Tankeev~\cite{TankK3} and Andr\'e~\cite{AndreK3}, and that by Proposition~\ref{prop:K3Maximal} the Hodge structure on~$H$ is always Hodge-maximal.

We retain the notation introduced in \ref{ssec:ModuliK3} and~\ref{ssec:K3Shim}. Choose a primitive polarization~$\lambda$ on~$Y$, say of degree~$2d$. Further choose an isometry $i \colon H \isomarrow L_{2d}$. These choices give us a compatible system $\bar{y} = (\bar{y}_K)$ of points $\bar{y}_K \in \cF_{2d,K}(\bbC)$, where $K$ runs through the set of open subgroups of~$K(3)$. Possibly after replacing~$k$ with a finite extension in~$\bbC$, we may assume that $k=k^\conn$ and that $\bar{y}_0 = \bar{y}_{K(3)}$ comes from a $k$-valued point $y_0 \in \cF_{2d,K(3)}(k)$ by composing it with the embedding $k \hookrightarrow \bbC$. Of course, $y_0$ is just the moduli point of $(Y,\lambda)$ equipped with a suitable level~$3$ structure.

Via the chosen isometry~$i$ the Hodge structure on~$H_\bbQ$ defines a point $h_0 \in \Omega^\pm$. Let $\bar{t} = (\bar{t}_K)$ be the system of $\bbC$-valued points $[h_0,eK]$ of $\Sh_K(\SO,\Omega^\pm)$, and abbreviate $\bar{t}_{K(3)}$ to~$\bar{t}_0$. The construction of the period map~\eqref{eq:PeriodMap} is such that $j_K(\bar{y}_K) = \bar{t}_K$ for all $K \subset K(3)$.

Let $t_0 = j_{K(3)}(y_0)$, which is a $k$-valued point of $\Sh_{K(3)}(\SO,\Omega^\pm)$. Let $\cF_0 \subset \cF_{2d,K(3)} \otimes k$ and $\cS_0 \subset \Sh_{K(3)}(\SO,\Omega^\pm) \otimes k$ be the irreducible components containing~$\bar{y}_0$ and~$\bar{t}_0$, respectively; as they are smooth over~$k$ and have a $k$-rational point, these components are geometrically irreducible. By construction, $j_0 = j_{K(3)}$ restricts to an \'etale morphism $j_0 \colon \cF_0 \to \cS_0$ over~$k$.

Consider the homomorphism $\phi_{\bar{t}} \colon \pi_1(\cS_0,\bar{t}_0) \to K(3)$ as in~\ref{ssec:ShimRepDef}. We also have the homomorphism $\Phi_{\bar{y}} \colon \pi_1(\cF_0,\bar{t}_0) \to K(3)$ of~\eqref{eq:PhiK3}. (In both cases we have now extended the base field to~$k$.) The fact that the diagrams~\eqref{eq:FtoSh} are Cartesian implies that $\Phi_{\bar{y}} = \phi_{\bar{t}} \circ j_{0,*}$. 

Let $H_\prim \subset H$ be the primitive integral cohomology, and identify the primitive \'etale cohomology with $\Zhat$-coefficients $\hat{H}_\prim \subset \hat{H}$ with $H_\prim \otimes \Zhat$. Via the chosen isometry~$i$, the Galois action on~$\hat{H}_\prim$ is a representation $\rho_{Y,\prim} \colon \Gal(\kbar/k) \to \SO(\Zhat)$. Note that the Galois action~$\rho_Y$ on~$\hat{H}$ leaves $\hat{H}_\prim$ stable and is trivial on the complement; hence the image of~$\rho_Y$ is the same as the image of~$\rho_{Y,\prim}$. On the other hand, the $k$-rational point~$y_0$ gives rise to a section~$\sigma_{y_0}$ of the homomorphism $\pi_1(\cF_0,\bar{y}_0) \to \Gal(\kbar/k)$ induced by the structural morphism $\cF_0 \to \Spec(k)$. The composition $\Phi_{\bar{y}} \circ \sigma_{y_0} \colon \Gal(\kbar/k) \to K(3) \subset \SO(\Zhat)$ is the same as~$\rho_{Y,\prim}$. The composition $j_{0,*} \circ \sigma_{y_0} \colon \Gal(\kbar/k) \to \pi_1(\cS_0,\bar{t}_0)$ is the section~$\sigma_{t_0}$ given by the point $t_0 \in \cS_0(k)$. It follows that $\phi_{\bar{t}} \circ \sigma_{t_0} \colon \Gal(\kbar/k) \to \SO(\Zhat)$ is the same as~$\rho_{Y,\prim}$.

The rest is the same as in the proof of Theorem~\ref{thm:MainThmAV}. Let $G_\Betti \subset \CSp_{2g}$ be the Mumford-Tate group, and let $X \subset \Omega^\pm$ be the $G_\Betti(\bbR)$-orbit of~$h_0$. The pair $(G_\Betti,X)$ is a Shimura datum and we have a morphism $f \colon (G_\Betti,X) \to (\SO,\Omega^\pm)$. Let $K_0 = f^{-1}\bigl(K(3)\bigr)$, let $E$ be the reflex field, and let $\Sh(f) \colon \Sh_{K_0}(G_\Betti,X) \to \Sh_{K(3)}(\SO,\Omega^\pm) \otimes E$ be the morphism induced by~$f$. We have a compatible system~$\bar{s}$ of points $\bar{s}_K = [h_0,eK] \in \Sh_K(G_\Betti,X)\bigl(\bbC\bigr)$ with $\Sh(f)\bigl(\bar{s}\bigr) = \bar{t}$. The point $\bar{s}_0 = \bar{s}_{K_0}$ comes from a $k$-valued point~$s_0$, and if $S_0 \subset \Sh_{K_0}(G_\Betti,X)_k$ is the irreducible component in which it lies, $s_0$ gives a section $\sigma_{s_0} \colon \Gal(\kbar/k) \to \pi_1(S_0,\bar{s}_0)$. Finally, if $\phi_{\bar{s}} \colon \pi_1(S_0,\bar{s}_0) \to K_0 \subset G_\Betti(\Zhat)$ is the representation~\eqref{eq:phi}, it follows from the functoriality explained in Remark~\ref{rem:phi} that $\phi_{\bar{s}} \circ \sigma_{s_0}$ is the same as~$\rho_{Y,\prim}$.

By the Mumford-Tate conjecture, $s_0$ is $\ell$-Galois-generic with respect to~$\phi_{\bar{s}}$ for every~$\ell$. By Theorem~\ref{thm:CadKret} it follows that $s_0$ is Galois-generic, and by Corollary~\ref{cor:Im(phi)}(\romannumeral3), the theorem follows.
\end{proof}